\documentclass{amsart}
\usepackage[english]{babel}
\usepackage{amssymb,bbm}

\catcode`\<=\active \def<{
\fontencoding{T1}\selectfont\symbol{60}\fontencoding{\encodingdefault}}
\catcode`\|=\active \def|{
\fontencoding{T1}\selectfont\symbol{124}\fontencoding{\encodingdefault}}
\newcommand{\assign}{:=}
\newcommand{\mathd}{\mathrm{d}}
\newcommand{\mathi}{\mathrm{i}}
\newcommand{\nobracket}{}
\newcommand{\nocomma}{}
\newcommand{\noplus}{}
\newcommand{\nosymbol}{}
\newcommand{\tmname}[1]{\textsc{#1}}
\newcommand{\tmop}[1]{\ensuremath{\operatorname{#1}}}
\newcommand{\tmtextbf}[1]{{\bfseries{#1}}}
\newcommand{\tmtextit}[1]{{\itshape{#1}}}
\newtheorem{proposition}{Proposition}
\newtheorem{theorem}{Theorem}

\begin{document}

\title{Discrete Hilbert Transform {\`a} la Gundy--Varopoulos}

\author{N. Arcozzi}
\address{Dip. di Matematica\\
Universita Bolo{\~n}a\\
Bolo{\~n}a, Italy}

\author{K. Domelevo}
\address{Inst. Math. Toulouse\\
Universit{\'e} Paul Sabatier\\
Toulouse, France}

\author{S. Petermichl}
\address{Inst. Math. Toulouse\\
Universit{\'e} Paul Sabatier\\
Toulouse, France}

\begin{abstract}
  We show that the centered discrete Hilbert transform on integers applied to
  a function can be written as the conditional expectation of a transform of
  stochastic integrals, where the stochastic processes considered have jump
  components. The stochastic representation of the function and that of its
  Hilbert transform are under differential subordination and orthogonality
  relation with respect to the sharp bracket of quadratic covariation. This
  illustrates the Cauchy Riemann relations of analytic functions in this
  setting. This result is inspired by the seminal work of Gundy and Varopoulos
  on stochastic representation of the Hilbert transform in the continuous
  setting.
\end{abstract}

{\maketitle}

{\bigskip}

\section{Introduction}

{\bigskip}

The subject of discrete analyticity and discrete Cauchy--Riemann relations has
been investigated for a long time. It originated in the works of
{\tmname{Ferrand}} {\cite{Fer1944a}} and {\tmname{Isaacs}} {\cite{Isa1941b}}.
The relationship of Cauchy--Riemann relations to a certain discrete Hilbert
transform was understood in {\tmname{Duffin}} {\cite{Duf1956a}} together with
the corresponding notions of discrete harmonic conjugate functions. The
discrete Hilbert transform also appeared in relation with the
{\tmname{Riesz}}--{\tmname{Titchmarsh}} transform as described in the
significant discovery by {\tmname{Matsaev}} {\cite{Mat1961a}}. See also
{\tmname{Matsaev}}--{\tmname{Sodin}} {\cite{MatSod2000a}}.

Despite the fact that (different notions of) discrete Hilbert transforms exist
for a long time, the precise $L^{p}$ norm of these discrete operators are
still a very famous open question. Optimal norm estimates are only known in
the continuous case -- see {\tmname{Pichorides}} {\cite{Pic1972}},
{\tmname{Verbitsky}} {\cite{Ver1980a}}, {\tmname{Essen}} {\cite{Ess1984}} --
whose proofs have a probabilistic reinterpretation, in part through the
formulae of {\tmname{Gundy}}-{\tmname{Varopoulos}} {\cite{GunVar1979}}.
It is remarkable that such representations allow one to prove sharp $L^p$ estimates
for discrete second order Riesz transforms as proved in \cite{DomPet2014c}
by using the Bellman technique and in \cite{ArcDomPet2016a} by using
stochastic tools. The versatility of these stochastic representations is
also seen in \cite{DomOsePet2017a} where various sharp estimates
for discrete second order Riesz transforms are proved.
Inspired by this fact, we aim at an understanding of a discrete Hilbert
transform through a stochastic integral formula, resembling the continuous
analog of {\tmname{Gundy}}-{\tmname{Varopoulos}}.

Regarding $L^{p}$--norm
estimates for the discrete Hilbert transform of the type we study below, to
the best of our knowledge the best constant is by {\tmname{Gohberg}} and
{\tmname{Krupnik}} {\cite{GohKre1970a}}. One important ingredient in the proof
of sharp estimates is the use of harmonic polynomials and other special
functions. There exist important discrepancies between harmonic functions and
polynomials in the continuous and discrete settings, which is part of the
reason why the sharp constants for the continuous case cannot be transfered in
a straightforward manner to the discrete case. We quote in that direction the
early work of {\tmname{Heilbronn}} {\cite{Hei1949a}} and the more recent works
{\cite{JerLevShe2014a}}{\cite{LipMan2015a}}{\cite{GuaMal2013a}}.

It is our main goal to prove a stocahstic representation formula for a discrete Hilbert
transform. One should keep in mind that there are different naive ways of
defining discrete Hilbert transforms from the space of sequences $\ell^{2} (
\mathbbm{Z} )$ onto itself. A first manner found in the litterature is to
mimic the continuous Hilbert transform $\mathcal{H}_{\mathbbm{R}}$ defined on
the real line as
\[ \forall x \in \mathbbm{R}, \hspace{1em} \mathcal{H}_{\mathbbm{R}} ( f ) ( x
   ) =- \frac{1}{\pi}   \int \frac{f ( x-y )}{y}   \mathd y, \]
where the integral is to be understood in the principal value sense. Indeed, a
naive discrete counterpart of $\mathcal{H}_{\mathbbm{R}}$ defined thanks to
the discrete convolution
\[ \mathcal{H}_{\mathbbm{Z}} : \ell^{2} ( \mathbbm{Z} ) \rightarrow \ell^{2} (
   \mathbbm{Z} ) , \hspace{1em} \forall x \in \mathbbm{Z}, \hspace{1em}
   \mathcal{H}_{\mathbbm{Z}} ( f ) ( x ) =- \frac{1}{\pi} \sum_{z \in
   \mathbbm{Z}_{\ast}} \frac{f ( x-z )}{z} \]
preserves the idea of a principal value integral by skipping $z=0$ in the
summation, but lacks many other important features of the continuous Hilbert
transform. For instance, the operator $\mathcal{H}_{\mathbbm{R}}$ is an
isometry in $L^{2} ( \mathbbm{R} )$, an anti-involution, i.e.
$\mathcal{H}_{\mathbbm{R}}^{2} \assign \mathcal{H}_{\mathbbm{R}} \circ
\mathcal{H}_{\mathbbm{R}} =- \tmop{Id}_{L^{2} ( \mathbbm{R} )}$ and obeys
$\mathcal{H}_{\mathbbm{R}} \circ \sqrt{- \Delta_{x}}  = \partial_{x}$.
However, $\mathcal{H}_{\mathbbm{Z}}$ above does not possess these basic
properties. One observes that the iterate $\mathcal{H}_{\mathbbm{Z}}^{2} =
\mathcal{H}_{\mathbbm{Z}} \circ \mathcal{H}_{\mathbbm{Z}}$ is far from being
neither an isometry of $\ell^{2} ( \mathbbm{Z} )$ nor an anti-involution. One
reason for that is the fact that the summation in the discrete convolution
excludes the integer $z=0$.

In this paper, we \ take the following route, as done in
{\tmname{Lust--Piquard}} {\cite{Lus2004a}}. Modelled after the defining
equation for the continuous Hilbert transform, $\mathcal{H}_{\mathbbm{R}}
\circ \sqrt{- \Delta_{x}}  = \partial_{x} \nocomma  $, let us recall the
definition of the discrete Laplacian on $\mathbbm{Z}:$ we define the discrete
derivatives as
\[ \hspace{1em} ( \partial_{x}^{+} f ) ( x ) \assign f ( x+1 ) -f ( x ) ,
   \hspace{1em} ( \partial_{x}^{-} f ) ( x ) \assign f ( x ) -f ( x-1 ) . \]
In $\ell^{2} \nocomma ( \mathbbm{Z} )$, it follows that $( \partial_{x}^{\pm}
)^{\star} =- \partial_{x}^{\mp}$, $\Delta_{x} = \partial_{x}^{+}
\partial_{x}^{-} = \partial_{x}^{-} \partial_{x}^{+} = \partial_{x}^{+} -
\partial_{x}^{-}$, and $( - \Delta_{x} ) = ( \partial_{x}^{\pm} )^{\star} (
\partial_{x}^{\pm} )$. Equipped with these discrete operators, another
classical definition for discrete Hilbert transforms is given by
\[ \mathcal{H}^{\pm} \circ \sqrt{- \Delta_{x}}  = \partial^{\pm}_{x} . \]
Those are the \tmtextbf{left} (resp. \tmtextbf{right}) discrete Hilbert
transform $\mathcal{H}^{+}$ (resp. $\mathcal{H}^{-}$), most often used when
defining Riesz transforms on discrete groups (see {\tmname{Lust--Piquard}}
{\cite{Lus2004a}} for more details and applications to the discrete Riesz
vector). Through explicit and simple calculations, we are going to see in the
next section that
\[ \mathcal{H}^{+} \mathcal{H}^{-} = \mathcal{H}^{-} \mathcal{H}^{+} =-
   \tmop{Id} , \hspace{2em} \| \mathcal{H}^{\pm} \|_{\ell^{2} ( \mathbbm{Z} )
   \rightarrow \ell^{2} ( \mathbbm{Z} )} =1 \]
and that the kernels of these operators $\mathcal{H}^{\pm}$ are
\[ - \frac{1}{\pi} \frac{1}{n \pm 1/2} \]
respectively.

The Fourier multiplier of $\mathcal{H}_{\mathbbm{R}}$ is constant on positive
and negative frequencies respectively. This is a feature the operators
$\mathcal{H}^{\pm}$ lack. In fact, these operators have Fourier multipliers
that are a modulation of the square wave function.

As mentioned earlier, another important and meaningful role is played by the
interplay of $\mathcal{H}_{\mathbbm{R}}$ and harmonic conjugate functions. If
$u^{f}$ is the Poisson extension of a function $f$ to the upper half plane and
$v^{f}$ the Poisson extension of $\mathcal{H}_{\mathbbm{R}} f$, then the pair
$u^{f}$ and $v^{f}$ obeys Cauchy-Riemann relations. Using space--time Brownian
motion, {\tmname{Gundy}} and {\tmname{Varopoulos}} have identified pairs of
martingales $M_{t}^{f}$ and $M^{\mathcal{H}_{\mathbbm{R}} f}_{t}$ that are
orthogonal and have differential subordination. In fact, in their formula,
$M_{t}^{\mathcal{H}_{\mathbbm{R}} f}$ is a martingale transform of \
$M_{t}^{f}$. The discrete counterpart of this feature of
$\mathcal{H}_{\mathbbm{R}}$ is the main focus of our note. Using stochastic
integrals driven by semidiscrete random walks in the (semidiscrete)
upper--half space, we will see that the \tmtextbf{centered} discrete Hilbert
transform defined as
\begin{equation}
  \mathcal{H} = \frac{1}{2} ( \mathcal{H}^{+} + \mathcal{H}^{-} ) \label{eq:
  definition discrete centered hilbert transform}
\end{equation}
enjoys this stochastic representation. Indeed, we obtain for $\mathcal{H}$
certain Cauchy-Riemann relations and an analog of the
{\tmname{Gundy}}--{\tmname{Varopoulos}} formula. In fact, $M_{t}^{\mathcal{H}
f}$ and $M_{t}^{f}$ are orthogonal and differentially subordinate with respect
to the sharp bracket $\langle \cdot , \cdot \rangle$.

The main goal of the paper is to prove the following representation formula
\tmtextit{{\`a} la} {\tmname{Gundy}}--{\tmname{Varopoulos}}
{\cite{GunVar1979}}:

\begin{theorem}
  \label{T: stochastic representation}\tmtextbf{(Stochastic representation)}
  The centered discrete Hilbert transform $\mathcal{H} f$ of a function $f \in
  \ell^{2} ( \mathbbm{Z} )$ as defined in (\ref{eq: definition discrete centered
  hilbert transform}) can be written as the conditional expectation
  \[ \forall x \in \mathbbm{Z}, \hspace{1em} \mathcal{H} f ( x ) =\mathbbm{E}
     ( N^{f}_{0}   |   \mathcal{Z}_{0} = ( x,0 ) ) \]
  where $N^{f}_{t}$, $- \infty <t \leqslant 0$, is a suitable martingale
  transform of a martingale $M_{t}^{f}$ associated to $f$, and
  $\mathcal{Z}_{t}$ is a suitable semidiscrete random walk on the semidiscrete
  upper--half space $\mathbbm{Z} \times \mathbbm{R}^{+}$.
\end{theorem}

\subsection*{Outline of the paper}The next section is devoted to a few basic
properties of the discrete hilbert transforms mentionned above. Section
\ref{S: Poisson extension and weak formulations} provides semidiscrete Poisson
extensions, weak formulations and semidiscrete Cauchy--Riemann relations. We
introduce the relevant stochastic integrals, martingale transforms and
quadratic covariations in Section \ref{S: stochastic representation}. Finally,
we prove the representation \tmtextit{{\`a} la}
{\tmname{Gundy}}--{\tmname{Varopoulos}} of the centered discrete Hilbert
transform in Section \ref{S: proof of the main result}.

{\bigskip}

\section{Basic properties}\label{S: basic properties}

{\bigskip}

Let $\mathcal{F}$ be the discrete fourier transform
\begin{eqnarray*}
  \mathcal{F} & : & \ell^{2} ( \mathbbm{Z} ) \longrightarrow L^{2} \left(
  \left] - \frac{1}{2} , \frac{1}{2} \right[ \right)\\
  &  & \mathcal{F} ( f ) ( \xi ) \assign \hat{f} ( \xi ) = \sum_{x \in
  \mathbbm{Z}} f ( x )  e^{- \mathi 2 \pi x \xi}
\end{eqnarray*}
with Fourier inverse
\begin{eqnarray*}
  \mathcal{F}^{-1} & : & L^{2} \left( \left] - \frac{1}{2} , \frac{1}{2}
  \right[ \right) \longrightarrow \ell^{2} ( \mathbbm{Z} )\\
  &  & \mathcal{F}^{-1} ( \hat{f} ) ( x ) = \int_{-1/2}^{1/2} \hat{f} ( \xi )
   e^{+ \mathi 2 \pi x \xi}   \mathd \xi
\end{eqnarray*}
Through explicit and simple calculations, we are going to see that

\begin{proposition}
  \label{P: equivalent definitions}\tmtextbf{(Equivalent definitions of
  $\mathcal{H}^{\pm}$)}
  \begin{equation}
    \widehat{\mathcal{H}^{\pm}} ( \xi ) =e^{\pm \mathi \pi \xi} \frac{\sin (
    \pi \xi )}{| \sin ( \pi \xi ) |} =e^{\mathi \pi \xi}   \tmop{SQ} ( \pi \xi
    )
  \end{equation}
  \begin{equation}
    \mathcal{H}^{\pm} f ( x ) =- \frac{1}{\pi} \sum_{z \in \mathbbm{Z}}
    \frac{f ( x-z )}{z \pm \tfrac{1}{2}}
  \end{equation}
\end{proposition}

As a consequence, the reader may check that we also have the following
properties,

\begin{proposition}
  \tmtextbf{(Basic properties)} The discrete Hilbert transforms
  $\mathcal{H}^{\pm}$ obey the following analogs of the continuous Hilbert
  transform
  \begin{equation}
    \forall f,g \in \ell^{2} ( \mathbbm{Z} ) , \hspace{1em} (
    \mathcal{H}^{\pm} f,g )_{\ell^{2} ( \mathbbm{Z} )} =- ( f,
    \mathcal{H}^{\mp} g )_{\ell^{2} ( \mathbbm{Z} )} , \hspace{1em} i.e
    \nosymbol .  ( \mathcal{H}^{\pm} )^{\star} =- \mathcal{H}^{\mp}
  \end{equation}
  \begin{equation}
    \mathcal{H}^{+} \mathcal{H}^{-} = \mathcal{H}^{-} \mathcal{H}^{+} =-
    \tmop{Id}
  \end{equation}
  \begin{equation}
    \mathcal{H}^{\pm} \mathcal{H}^{\pm} =-S_{\pm 1} \tmop{Id}
  \end{equation}
  \begin{equation}
    \| \mathcal{H}^{\pm} \|_{\ell^{2} ( \mathbbm{Z} ) \rightarrow \ell^{2} (
    \mathbbm{Z} )} =1
  \end{equation}
  where $S_{+1}$ (resp. $S_{-1}$) is the right (resp. left) shift operator $(
  S_{\pm 1} f ) ( x ) =f ( x \mp 1 )$.
\end{proposition}

\begin{proof}
  \tmtextbf{(of Proposition \ref{P: equivalent definitions})} One has
  succesively
  \begin{eqnarray*}
    \widehat{\partial_{x}^{\pm}} ( \xi ) & = & e^{\pm \mathi 2 \pi \xi} -1=2
    \mathi  e^{\pm \mathi \pi \xi}   \sin ( \pi \xi )\\
    \widehat{\Delta_{x}} ( \xi ) & = & -4  \sin^{2} ( \pi \xi )\\
    \widehat{\sqrt{- \Delta_{x}}} ( \xi ) & = & 2  | \sin ( \pi \xi ) |\\
    \widehat{\mathcal{H}^{\pm}} ( \xi ) & = & \mathi  e^{\pm \mathi \pi \xi}
    \frac{\sin ( \pi \xi )}{| \sin ( \pi \xi ) |} =: \mathi  e^{\mathi \pi
    \xi}  \tmop{SQ} ( \pi \xi ) .
  \end{eqnarray*}
  Now, computing the Fourier transform of the discrete kernel, we check
  \begin{eqnarray*}
    \sum_{y \in \mathbbm{Z}} - \frac{1}{\pi} \frac{1}{y+ \tfrac{1}{2}} e^{-2
    \pi \mathi y \xi} & = & - \frac{2}{\pi} \sum_{y} \frac{e^{- \mathi ( 2y+1
    ) \tfrac{\xi}{2}}}{2y+1} e^{2 \pi \mathi \tfrac{\xi}{2}}\\
    &  & \\
    & = & - \frac{2}{\pi} \mathi e^{\pi \mathi \xi} \sum_{m} \frac{\sin
    \left( -2 \pi ( 2y+1 ) \tfrac{\xi}{2} \right)}{2y+1}\\
    &  & \\
    & = & - \frac{4}{\pi} \mathi e^{2 \pi \mathi \tfrac{\xi}{2}} \sum_{m
    \geqslant 0} \frac{\sin \left( \left. 2 \pi ( 2y+1 ) \left( -
    \tfrac{\xi}{2} \right. \right) \right)}{2y+1}\\
    &  & \\
    & = & - \mathi e^{\pi \mathi \xi \frac{}{}} \tmop{SQ} \left( - \frac{\pi
    \xi}{2} \right)\\
    &  & \\
    & = & \mathi e^{\pi \mathi \xi} \frac{\sin ( \pi \xi )}{| \sin ( \pi \xi
    ) |} ,
  \end{eqnarray*}
  where we used the Fourier transform of the square wave function
  \[ \tmop{SQ} ( t ) = \frac{4}{\pi} \sum^{\infty}_{n =0} \frac{\sin \left(
     \tfrac{2 \pi}{T} ( 2n+1 ) t \right)}{2n+1} . \]
  Similarly, the symbol of $\mathcal{H}^{-}$ is $\mathi e^{- \pi \mathi \xi}
  \frac{\sin ( \pi \xi )}{| \sin ( \pi \xi ) |} =- \mathi  e^{- \pi i \xi
  } \tmop{SQ} \left( - \frac{\xi}{2} \right)$
\end{proof}

Finally, for the centered discrete Hilbert transform, one has
\begin{eqnarray}
  \mathcal{H} \circ \mathcal{H} & = & [ ( \mathcal{H}^{+} )^{2} +2 (
  \mathcal{H}^{+} ) ( \mathcal{H}^{-} ) + ( \mathcal{H}^{-} )^{2} ] /4 
  \label{eq: H*H}\\
  & = & [ ( \mathcal{H}^{+} )^{2} +2 ( \mathcal{H}^{+} ) ( \mathcal{H}^{-} )
  + ( \mathcal{H}^{-} )^{2} ] /4 \nonumber\\
  & = & [ ( \mathcal{H}^{+} ) S^{+} ( \mathcal{H}^{-} ) +2 ( \mathcal{H}^{+}
  ) ( \mathcal{H}^{-} ) +S^{-} ( \mathcal{H}^{+} ) ( \mathcal{H}^{-} ) ] /4
  \nonumber\\
  & = & - \left( \frac{1}{4} S_{-} + \frac{1}{2} \tmop{Id} + \frac{1}{4}
  S_{+} \right) , \nonumber
\end{eqnarray}
where we recognize a smoothed version of minus the identity.

{\bigskip}

\section{Semi-discrete Poisson extensions and weak formulations}\label{S:
Poisson extension and weak formulations}

{\bigskip}

\subsection*{Semidiscrete Poisson extensions}Defining the selfadjoint operator
$A$ as the square root $A= \sqrt{- \Delta_{x}}$, we set $P_{y} =e^{-y A}$, $y
\in [ 0, \infty [$. The Poisson extension of a function $f \in \ell^{2} (
\mathbbm{Z} )$ is the function $f ( y,x ) \assign ( P_{y} f ) ( x )$ defined
on $\mathbbm{R}^{+} \times \mathbbm{Z}$. Note that we use the same name for
$f$ and it's Poisson extension. It follows that the function $f ( y,x )$
satisfies in $( 0, \infty ) \times \mathbbm{Z}$
\[ \partial_{y} f ( x,y ) =-A f ( x,y ) , \hspace{1em} \partial^{2}_{y} f (
   x,y ) =A^{2} f ( x,y ) =- \Delta_{x} f ( x,y ) \]
that is $f ( x,y )$ is harmonic in $( 0, \infty ) \times \mathbbm{Z}$:
\[ ( \Delta_{y} + \Delta_{x} ) f=0. \]
For any $f$ and $g$ in $\ell^{2} ( \mathbbm{Z} )$, we note
\[ ( f,g )_{\ell^{2} ( \mathbbm{Z} )} \assign \sum_{x \in \mathbbm{Z}} f ( x ) g
   ( x ) \]
the scalar product of $\ell^{2} ( \mathbbm{Z} )$. We will often omit the
subscript in the scalar product and write $( f,g )$ instead of $( f,g )_{\ell^{2}
( \mathbbm{Z} )}$. Moreover we have
\begin{eqnarray*}
  \forall f,g \in \ell^{2} ( \mathbbm{Z} ) , \hspace{1em} ( - \Delta_{x} f,g ) &
  \assign & ( ( \partial_{x}^{\pm} )^{\star} ( \partial_{x}^{\pm} ) f,g ) = (
  \partial_{x}^{\pm} f, \partial_{x}^{\pm} g )
\end{eqnarray*}
and also
\begin{eqnarray*}
  \forall f,g \in \ell^{2} ( \mathbbm{Z} ) , \hspace{1em} ( - \Delta_{x} f,g ) &
  = & ( A^{2} f,g ) = ( A f,A g ) = ( \partial_{y} f, \partial_{y} g )
\end{eqnarray*}
We will collect below all the derivations in the $4$--vector
\[ \nabla_{y,x} = ( \partial_{y} , \partial_{x}^{+} , \partial_{y} ,
   \partial_{x}^{-} )^{\star} . \]
Notice that we repeat twice the derivation $\partial_{y}$. The reasons for
that will become clear later.

\begin{theorem}
  \tmtextbf{(Weak formulation of the identity operator)} Assume $f$ and $g$ in
  $\ell^{2} ( \mathbbm{Z} )$. Let $\mathcal{I}$ denote the identity operator. We
  have the Littlewood--Paley identity
  \[ ( \mathcal{I} f,g ) = \int_{y=0}^{\infty} \sum_{x \in \mathbbm{Z}} (
     \nabla_{x,y} f ( x,y ) , \nabla_{x,y} g ( x,y ) )  y  \mathd y \]
\end{theorem}

\begin{proof}
  Notice first that for any functions $f$ and $g$ in $\ell^{2} ( \mathbbm{Z} )$,
  we have, using successively discrete integration by parts, and the
  definition of $A$,
  \begin{eqnarray*}
    ( - \Delta_{x} f,g ) & \assign & ( ( \partial_{x}^{\pm} )^{\star} (
    \partial_{x}^{\pm} ) f,g ) = ( \partial_{x}^{\pm} f, \partial_{x}^{\pm} g
    )\\
    & = & ( A^{2} f,g ) = ( A f,A g )
  \end{eqnarray*}
  In the particular case where both $f=f ( x,y )$ and $g=g ( x,y )$ are
  Poisson extensions, then
  \[ ( A f,A g ) = ( \partial_{y} f, \partial_{y} g ) \]
  To summarize, when the functions $f ( y ) =P_{y} f$ and $g ( y ) =P_{y} g$
  are Poisson extensions, we have in the upper half space
  \[ ( - \Delta_{x} f ( y ) ,g ( y ) ) = ( \partial_{x}^{\pm} f ( y ) ,
     \partial_{x}^{\pm} g ( y ) ) = ( \partial_{y} f ( y ) , \partial_{y} g (
     y ) ) = ( A f ( y ) ,A g ( y ) ) . \]
  But for any function $F ( y )$ smooth enough and decaying at infinity, we
  have
  \[ F ( 0 ) = \int_{0}^{\infty} F'' ( y ) y \mathd y \]
  Applying this identity to $F ( y ) = ( f ( y ) ,g ( y ) )_{\ell^{2} (
  \mathbbm{Z} )}$ yields,
  \begin{eqnarray*}
    ( f,g )_{\ell^{2} ( \mathbbm{Z} )} & = & \int_{0}^{\infty} \{ (
    \partial^{2}_{y} f,g ) +2 (   \partial_{y} f, \partial_{y} g ) + ( f,
    \partial^{2}_{y} g ) \} y  \mathd y\\
    & = & 4 \int_{0}^{\infty} ( \partial_{y} f, \partial_{y} g )  y  \mathd
    y\\
    & = & 4 \int_{0}^{\infty} ( \partial_{x}^{+} f, \partial_{x}^{+} g ) y 
    \mathd y=4 \int_{0}^{\infty} ( \partial_{x}^{-} f, \partial_{x}^{-} g )  y
    \mathd y\\
    & = & \int_{0}^{\infty} \{ ( \partial_{y} f, \partial_{y} g ) + (
    \partial_{x}^{+} f, \partial_{x}^{+} g ) + ( \partial_{y} f, \partial_{y}
    g ) + ( \partial_{x}^{-} f, \partial_{x}^{-} g ) \}  y  \mathd y\\
    & = & \int_{0}^{\infty} ( \nabla_{y,x} f, \nabla_{y,x} g )  y  \mathd y,
  \end{eqnarray*}
  as announced.
\end{proof}

\subsection*{Cauchy--Riemman relations and weak formulation}

\begin{theorem}
  \label{T: Cauchy-Riemann}\tmtextbf{(Cauchy-Riemann relations)} Let $f$ and
  $g$ in $\ell^{2} ( \mathbbm{Z} )$. Let $\mathcal{H} f ( y,x )$ denote the
  Poisson extension of $\mathcal{H} f ( x )$, and $f ( y,x )$ that of $f ( x
  )$. We have the semidiscrete Cauchy--Riemman relations, for all $( y,x ) \in
  R^{+} \times \mathbbm{Z}$,

  \[ \partial_{y} \mathcal{H}^{\pm} f=- \partial^{\mp}_{x} f, \hspace{1em}
     \partial_{x}^{\pm} \mathcal{H}^{\pm} f= \partial_{y} f, \label{eq:
     Cauchy-Riemann} \]
  which implies for the centered discrete Hilbert transform, for all $( y,x )
  \in R^{+} \times \mathbbm{Z}$,
  \[ \partial_{y} \mathcal{H} f=- \partial_{x}^{0} f, \hspace{1em}
     \partial_{x}^{0} \mathcal{H} f= \left( \frac{1}{4} S_{-} + \frac{1}{2}
     \tmop{Id} + \frac{1}{4} S_{+} \right) \partial_{y} f. \]
\end{theorem}

\begin{theorem}
  \label{T: weak formulation}\tmtextbf{(Weak formulation for the discrete
  Hilbert)} Let $\mathcal{H} f$ denote the centered discrete Hilbert transform
  of $f$. Let $f \assign f ( y,x )$ denote the Poisson extension of $f$. Let
  $g \assign g ( y,x )$ denote the Poisson extension of a test function $g$.
  We have the weak formulation:
  \begin{eqnarray*}
    ( \mathcal{H} f,g ) & = & \int_{0}^{\infty} ( A \nabla_{y,x} f,
    \nabla_{y,x} g )  y  \mathd y \label{eq: weak formulation Hilbert
    transform}
  \end{eqnarray*}
  where $A \in \mathbbm{R}^{4 \times 4}$ is the matrix
  \[ A= \left(\begin{array}{cccc}
       0  & 0 & 0 & -1\\
       0 & 0 & 1 & 0\\
       0 & -1 & 0 & 0\\
       1 & 0 & 0 & 0
     \end{array}\right) . \]
\end{theorem}

It is important to observe that the weak formulation involves the orthogonal
matrix $A$ such that $A^{2} =- \tmop{Id}$. This does not mean that the
centered discrete Hilbert transform $\mathcal{H}$ is an antiinvolution, as is
clear from Theorem \ref{T: Cauchy-Riemann} and equation (\ref{eq: H*H}).

\begin{proof}
  The semidiscrete Poisson extension convenient for us has up to normalization
  the kernel
  \[ P_{y} ( x ) = \int_{0}^{1} e^{-2 \pi \mathi \xi x} e_{}^{-2 | \sin ( \pi
     \xi ) | y} d \xi \]
  Let us write for convenience $u$ the Poisson extension of $f$
  \[ u ( x_{} ,y ) = \int_{0}^{1} \hat{f} ( \xi ) e^{-2 \pi \mathi \xi x}
     e_{}^{-2 | \sin ( \pi \xi ) | y} d \xi \]
  and $v^{\pm}$ the Poisson extension of $\mathcal{H}^{\pm} f$
  \begin{eqnarray*}
    v^{\pm} ( x,y ) & = & \int_{0}^{1} \widehat{\mathcal{H}_{\pm} f} ( \xi )
    e^{-2 \pi \mathi \xi x} e_{}^{-2 | \sin ( \pi \xi ) | y} d \xi\\
    & = & - \mathi \int_{0}^{1} e^{\noplus \pm \pi \mathi \xi} \frac{\sin (
    \pi \xi )}{| \sin ( \pi \xi ) |} \hat{f} ( \xi ) e^{-2 \pi \mathi \xi x}
    e_{}^{-2 | \sin ( \pi \xi ) | y} d \xi
  \end{eqnarray*}
  Observe that
  \begin{eqnarray*}
    \partial^{+}_{x} u ( x,y ) & = & \int_{0}^{1} \hat{f} ( \xi ) ( e^{-2 \pi
    \mathi \xi ( x+1 )} -e^{-2 \pi \mathi \xi x} ) e_{}^{-2 | \sin ( \pi \xi )
    | y} d \xi\\
    & = & -2 \mathi \int_{0}^{1} e^{- \pi \mathi \xi} \sin ( \pi \xi )
    \hat{f} ( \xi ) e^{-2 \pi \mathi \xi x } e_{}^{-2 | \sin ( \pi \xi ) | y}
    d \xi
  \end{eqnarray*}
  While
  \begin{eqnarray*}
    \partial_{x}^{-} u ( x,y ) & = & \int_{0}^{1} \hat{f} ( \xi ) ( e^{-2 \pi
    \mathi \xi x} -e^{-2 \pi \mathi \xi ( x-1 )} ) e_{}^{-2 | \sin ( \pi \xi )
    | y} d \xi\\
    &  & \\
    & = & -2 \mathi \int_{0}^{1} e^{\pi \mathi \xi} \sin ( \pi \xi ) \hat{f}
    ( \xi ) e^{-2 \pi \mathi \xi x } e_{}^{-2 | \sin ( \pi \xi ) | y} d \xi
  \end{eqnarray*}
  Similarly
  \[ \partial_{x}^{+} v^{+} ( x,y ) =-2 \int_{0}^{1} | \sin ( \pi \xi ) |
     \hat{f} ( \xi ) e^{-2 \pi \mathi \xi x} e_{}^{-2 | \sin ( \pi \xi ) | y}
     d \xi \]
  and

  \[ \partial_{x}^{-} v^{-} ( x,y ) =-2 \int_{0}^{1} | \sin ( \pi \xi ) |
     \hat{f} ( \xi ) e^{-2 \pi \mathi \xi x} e_{}^{-2 | \sin ( \pi \xi ) | y}
     d \xi \]
  The continuous derivatives in the other variable are
  \[ \partial_{y} u ( x,y ) =-2 \int_{0}^{1} | \nobracket \sin ( \pi \xi ) |
     \nobracket \hat{f} ( \xi ) e^{-2 \pi \mathi \xi x} e_{}^{-2 | \sin ( \pi
     \xi ) | y} d \xi \]
  and
  \[ \partial_{y} v^{\pm} ( x,y ) =2 \mathi \int_{0}^{1} e^{\pm \pi \mathi
     \xi} \sin ( \pi \xi ) \hat{f} ( \xi ) e^{-2 \pi \mathi \xi x} e_{}^{-2 |
     \sin ( \pi \xi ) | y} d \xi    \]
  This gives the following Cauchy Riemann equations
  \[ \partial_{y} v^{+} =- \partial^{-}_{x} u \nocomma , \]
  \[ \partial_{y} v^{-} =- \partial_{x}^{+} u \]
  but
  \[ \partial_{y} u= \partial_{x}^{+} v^{+} = \partial_{x}^{-} v^{-} . \]
  Let us note $u^{g}$ the Poisson extension of a test funciton $g$ and
  accordingly $u=u^{f}$ the Poisson extension of $f$. We have successively
  \begin{eqnarray*}
    ( \mathcal{H}^{+} f,g ) & = & 2 \int_{0}^{\infty} ( \partial_{y}  v^{+} (
    y ) , \partial_{y}  u^{g} ( y ) ) y  \mathd y+2 \int_{0}^{\infty} (
    \partial_{x}^{+} v^{+} ( y ) , \partial_{x}^{+}  u^{g} ( y ) ) y  \mathd
    y\\
    & = & 2 \int_{0}^{\infty} ( - \partial_{x}^{-} u^{f} ( y ) , \partial_{y}
    u^{g} ( y ) ) y  \mathd y+2 \int_{0}^{\infty} ( \partial_{y} u^{f} ( y )
    , \partial_{x}^{+}  u^{g} ( y ) ) y  \mathd y\\
    ( \mathcal{H}^{-} f,g ) & = & 2 \int_{0}^{\infty} ( - \partial_{x}^{+}
    u^{f} ( y ) , \partial_{y}  u^{g} ( y ) ) y  \mathd y+2 \int_{0}^{\infty}
    ( \partial_{y} u^{f} ( y ) , \partial_{x}^{-}  u^{g} ( y ) ) y  \mathd y\\
    ( \mathcal{H} f,g ) & = & \int_{0}^{\infty} ( - \partial_{x}^{-} u^{f} ( y
    ) , \partial_{y}  u^{g} ( y ) ) y  \mathd y+ \int_{0}^{\infty} (
    \partial_{y} u^{f} ( y ) , \partial_{x}^{+}  u^{g} ( y ) ) y  \mathd y\\
    &  & \hspace{1em} + \int_{0}^{\infty} ( - \partial_{x}^{+} u^{f} ( y ) ,
    \partial_{y}  u^{g} ( y ) ) y  \mathd y+ \int_{0}^{\infty} ( \partial_{y}
    u^{f} ( y ) , \partial_{x}^{-}  u^{g} ( y ) ) y  \mathd y\\
    & = & \int_{0}^{\infty} \left( \left(\begin{array}{r}
      - \partial^{-}_{x}\\
      \partial_{y}\\
      - \partial^{+}_{x}\\
      \partial_{y}
    \end{array}\right) u^{f} ( y ) , \left(\begin{array}{l}
      \partial_{y}\\
      \partial^{+}_{x}\\
      \partial_{y}\\
      \partial_{x}^{-}
    \end{array}\right) u^{g} ( y ) \right) y  \mathd y\\
    & = & \int_{0}^{\infty} \left( \left(\begin{array}{cccc}
      0  & 0 & 0 & -1\\
      0 & 0 & 1 & 0\\
      0 & -1 & 0 & 0\\
      1 & 0 & 0 & 0
    \end{array}\right) \left(\begin{array}{l}
      \partial_{y}\\
      \partial^{+}_{x}\\
      \partial_{y}\\
      \partial_{x}^{-}
    \end{array}\right) u^{f} ( y ) , \left(\begin{array}{l}
      \partial_{y}\\
      \partial^{+}_{x}\\
      \partial_{y}\\
      \partial_{x}^{-}
    \end{array}\right) u^{g} ( y ) \right) y  \mathd y
  \end{eqnarray*}
  This concludes the proof of Theorems \ref{T: Cauchy-Riemann} and \ref{T:
  weak formulation}.
\end{proof}

{\bigskip}

\section{Stochastic representations}\label{S: stochastic representation}

{\bigskip}

\subsection*{Semidiscrete random walks}Let $\mathcal{N}_{t}$ be a
c{\`a}dl{\`a}g Poisson process with parameter $\lambda$. Let $( T_{k} )_{k \in
\mathbbm{N}}$ be the instants of jumps. Let $( \varepsilon_{k} )_{k \in
\mathbbm{N}}$ be a sequence of independent Bernouilli variables,
\[ \forall k, \hspace{1em} \mathbbm{P} ( \varepsilon_{k} =1 ) =\mathbbm{P} (
   \varepsilon_{k} =-1 ) =1/2. \]
This allows us to define the random walk $\mathcal{X}_{t} \in \mathbbm{Z}$ as
the compound Poisson process (see e.g. {\tmname{Protter}} {\cite{Pro2005a}},
{\tmname{Privault}} {\cite{Pri2009a}}{\cite{Pri2014a}})
\[ \mathcal{X}_{0} \mathcal{\mathcal{}} \in \mathbbm{Z}, \hspace{1em}
   \mathcal{X}_{t} = \sum_{k=1}^{\mathcal{N}_{t}} \varepsilon_{k} \]
Let $\mathcal{Y}_{t}$ be a standard onedimensional brownian process started at
$\mathcal{Y}_{0}$. We define the semidiscrete random walk $\mathcal{Z}_{t}$ as
$\mathcal{Z}_{t} \assign ( \mathcal{Y}_{t} , \mathcal{X}_{t} ) \in
\mathbbm{R}^{+} \times \mathbbm{Z}$, i.e.
\[ \mathcal{Z_{0} = ( \mathcal{Y}_{0} , \mathcal{X}_{0} )} , \hspace{1em}
   \mathd \mathcal{Z}_{t} = ( \mathd B_{t} , \varepsilon_{\mathcal{N}_{t}}  
   \mathd \mathcal{N}_{t} ) \]

\subsection*{Stochastic integrals}Let $f$ defined on $\mathbbm{R}^{+} \times
\mathbbm{Z}$ be a smooth function. We have the It{\^o} formula
\begin{eqnarray*}
  f ( \mathcal{Z}_{t} ) -f ( \mathcal{Z}_{0} ) & = & \int_{0}^{t}
  \partial_{x}^{0} f ( \mathcal{Z}_{s_{-}} ) \mathd \mathcal{X}_{s} +
  \int_{0}^{t} \frac{1}{2} \partial_{x}^{2} f ( \mathcal{Z}_{s_{-}} ) \mathd (
  \mathcal{N}_{s} -s )\\
  &  & \hspace{2em} + \int_{0}^{t} \partial_{y} f ( \mathcal{Z}_{s_{-}} )
  \mathd \mathcal{Y}_{s} + \left\{ \frac{1}{2} \int_{0}^{t} ( \partial^{2}_{x}
  f+ \partial^{2}_{y} f ) ( \mathcal{Z}_{s_{-}} ) \mathd s \right\}
\end{eqnarray*}
This formula can be derived thanks to telescopic sums involving jump times. We
refer to {\tmname{Privault}} {\cite{Pri2009a}}{\cite{Pri2014a}} for more
details.

\subsection*{Quadratic covariation} Let $f$ and $g$ be two semidiscrete harmonic
functions in $\mathbbm{R}^{+}_{\star} \times \mathbbm{Z}$, that is
\[ \partial^{2}_{x} f+ \partial^{2}_{y} f= \partial^{2}_{x} g+
   \partial^{2}_{y} g=0 \hspace{1em} \tmop{in} \hspace{1em}
   \mathbbm{R}^{+}_{\star} \times \mathbbm{Z}. \]
We define the corresponding martingales $M_{t}^{f} \assign f ( \mathcal{Z}_{t}
)$ and $M_{t}^{g} \assign g ( \mathcal{Z}_{t} )$, so that
\begin{eqnarray*}
  \mathd M_{t}^{f} & = & \partial_{x}^{0} f ( \mathcal{Z}_{t_{-}} )   \mathd
  \mathcal{X}_{t} + \frac{1}{2} \partial_{x}^{2} f ( \mathcal{Z}_{t_{-}} )  
  \mathd ( \mathcal{N}_{t} -t ) + \partial_{y} f ( \mathcal{Z}_{t_{-}} )  
  \mathd \mathcal{Y}_{t}\\
  \mathd M_{t}^{g} & = & \partial_{x}^{0} g ( \mathcal{Z}_{t_{-}} )   \mathd
  \mathcal{X}_{t} + \frac{1}{2} \partial_{x}^{2} g ( \mathcal{Z}_{t_{-}} )  
  \mathd ( \mathcal{N}_{t} -t ) + \partial_{y} g ( \mathcal{Z}_{t_{-}} )  
  \mathd \mathcal{Y}_{t}
\end{eqnarray*}
It follows that
\begin{eqnarray*}
  \mathd [ M^{f} ,M^{g} ]_{t} & = & \partial_{x}^{0} f ( \mathcal{Z}_{t_{-}} )
  \partial_{x}^{0} g ( \mathcal{Z}_{t_{-}} ) \mathd [ \mathcal{X} ,
  \mathcal{X} ]_{t} + \frac{1}{4} \partial_{x}^{2} f ( \mathcal{Z}_{t_{-}} )  
  \partial_{x}^{2} g ( \mathcal{Z}_{t_{-}} ) \mathd [ \mathcal{N, \mathcal{N}}
  ]_{t}\\
  &  & \hspace{2em} + \left( \partial_{x}^{0} f ( \mathcal{Z}_{t_{-}} )
  \frac{1}{2} \partial_{x}^{2} g ( \mathcal{Z}_{t_{-}} ) + \partial_{x}^{0} g
  ( \mathcal{Z}_{t_{-}} ) \frac{1}{2} \partial_{x}^{2} f ( \mathcal{Z}_{t_{-}}
  ) \right) \mathd [ \mathcal{X} , \mathcal{N} ]_{t}\\
  &  & \hspace{2em} + \partial_{y} f ( \mathcal{Z}_{t_{-}} ) \partial_{y} g (
  \mathcal{Z}_{t_{-}} ) \mathd [ \mathcal{Y, \mathcal{Y}} ]_{t}\\
  & = & \left[ \partial_{x}^{0} f ( \mathcal{Z}_{t_{-}} ) \partial_{x}^{0} g
  ( \mathcal{Z}_{t_{-}} ) + \frac{1}{2} \partial_{x}^{2} f (
  \mathcal{Z}_{t_{-}} ) \frac{1}{2} \partial_{x}^{2} g ( \mathcal{Z}_{t_{-}} )
  \right] \mathd \mathcal{N}_{t}\\
  &  & \hspace{2em} + \left[ \partial_{x}^{0} f ( \mathcal{Z}_{t_{-}} )
  \frac{1}{2} \partial_{x}^{2} g ( \mathcal{Z}_{t_{-}} ) + \partial_{x}^{0} g
  ( \mathcal{Z}_{t_{-}} ) \frac{1}{2} \partial_{x}^{2} f ( \mathcal{Z}_{t_{-}}
  ) \right] \mathd \mathcal{X}_{t}\\
  &  & \hspace{2em} + \partial_{y} f ( \mathcal{Z}_{t_{-}} ) \partial_{y} g (
  \mathcal{Z}_{t_{-}} ) \mathd t\\
  & = & \frac{1}{2} [ \partial_{x}^{+} f  \partial_{x}^{+} g+
  \partial_{x}^{-} f  \partial_{x}^{-} g ] \mathd \mathcal{N}_{t}\\
  &  & \hspace{2em} + \frac{1}{2} [ \partial_{x}^{+} f  \partial_{x}^{+} g-
  \partial_{x}^{-} f  \partial_{x}^{-} g ] \mathd \mathcal{X}_{t}\\
  &  & \hspace{2em} + \partial_{y} f ( \mathcal{Z}_{t_{-}} ) \partial_{y} g (
  \mathcal{Z}_{t_{-}} ) \mathd t\\
  & = & [ ( \Delta \mathcal{X}_{t} )_{+}   \partial_{x}^{+} f 
  \partial_{x}^{+} g+ ( \Delta \mathcal{X}_{t} )_{-}   \partial_{x}^{-} f 
  \partial_{x}^{-} g ] \mathd \mathcal{N}_{t}\\
  &  & \hspace{2em} + \partial_{y} f ( \mathcal{Z}_{t_{-}} )   \partial_{y} g
  ( \mathcal{Z}_{t_{-}} ) \mathd t
\end{eqnarray*}
where $( \Delta \mathcal{X} )_{\pm} \assign \max ( 0, \pm \Delta \mathcal{X}
)$.

\subsection*{Martingale transform}In order to define a martingale transform
$N_{t}^{f}$ of $M_{t}^{f}$ such that $N_{t}^{f}$ allows us to recover the
discrete hilbert transform, we recall that the weak formulation (\ref{eq: weak
formulation Hilbert transform}) of Theorem \ref{T: weak formulation} involves
the ``martingale transform''
\begin{eqnarray*}
  \nabla_{y,x} = \left(\begin{array}{c}
    \partial_{y}^{-}\\
    \partial_{x}^{+}\\
    \partial_{y}^{+}\\
    \partial_{x}^{-}
  \end{array}\right) & \longrightarrow & \nabla_{y,x}^{\perp} =
  \left(\begin{array}{r}
    - \partial_{x}^{-}\\
    \partial_{y}^{+}\\
    - \partial_{x}^{+}\\
    \partial_{y}^{-}
  \end{array}\right) .
\end{eqnarray*}
Let us first rewrite the martingale increments in terms of the
$\partial_{x}^{\pm}$ derivatives:
\begin{eqnarray*}
  \mathd M_{t}^{f} & = & \partial_{x}^{0} f ( \mathcal{Z}_{t_{-}} )   \mathd
  \mathcal{X}_{t} + \frac{1}{2} \partial_{x}^{2} f ( \mathcal{Z}_{t_{-}} )  
  \mathd ( \mathcal{N}_{t} -t ) + \partial_{y} f ( \mathcal{Z}_{t_{-}} )  
  \mathd \mathcal{Y}_{t}\\
  & = & \frac{1}{2} \partial_{x}^{+} f ( \mathcal{Z}_{t_{-}} )   ( \mathd
  \mathcal{X}_{t} + \mathd ( \mathcal{N}_{t} -t ) ) + \frac{1}{2}
  \partial_{x}^{-} f ( \mathcal{Z}_{t_{-}} )   ( \mathd \mathcal{X}_{t} -
  \mathd ( \mathcal{N}_{t} -t ) )\\
  &  & \hspace{2em} + \frac{1}{2} \partial^{-}_{y} f ( \mathcal{Z}_{t_{-}} ) 
  \mathd \mathcal{Y}_{t} + \frac{1}{2} \partial^{+}_{y} f (
  \mathcal{Z}_{t_{-}} )   \mathd \mathcal{Y}_{t} .
\end{eqnarray*}
The Cauchy-Riemann relations therefore suggest to define
\begin{eqnarray*}
  \mathd N_{t}^{f} & \assign & \frac{1}{2} \partial_{y}^{+} f (
  \mathcal{Z}_{t_{-}} )   ( \mathd \mathcal{X}_{t} + \mathd ( \mathcal{N}_{t}
  -t ) ) + \frac{1}{2} \partial_{y}^{-} f ( \mathcal{Z}_{t_{-}} )   ( \mathd
  \mathcal{X}_{t} - \mathd ( \mathcal{N}_{t} -t ) )\\
  &  & \hspace{2em} + \frac{1}{2} ( - \partial^{-}_{x} f ) (
  \mathcal{Z}_{t_{-}} )   \mathd \mathcal{Y}_{t} + \frac{1}{2} ( -
  \partial^{+}_{x} f ) ( \mathcal{Z}_{t_{-}} )   \mathd \mathcal{Y}_{t}\\
  & = & \partial_{y} f ( \mathcal{Z}_{t_{-}} )   \mathd \mathcal{X}_{t} - (
  \partial^{0}_{x} f ) ( \mathcal{Z}_{t_{-}} )   \mathd \mathcal{Y}_{t}
\end{eqnarray*}
To summarize, let $f$ be harmonic. We have defined $N_{t}^{f}$ as the
stochastic integral
\[ N_{t}^{f} \assign f ( \mathcal{Z}_{0} ) + \int_{0}^{t} \partial_{y} f (
   \mathcal{Z}_{s_{-}} )   \mathd \mathcal{X}_{s} - \partial^{0}_{x} f (
   \mathcal{Z}_{s_{-}} )   \mathd \mathcal{Y}_{s} . \]
It is now easy to estimate the quadratic covariation
\begin{eqnarray*}
  \mathd [ M^{f} ,N^{f} ]_{t} & = & \partial_{x}^{0} f ( \mathcal{Z}_{t_{-}} )
  \partial_{y} f ( \mathcal{Z}_{t_{-}} )   \mathd [ \mathcal{X} ,
  \mathcal{X} ]_{t} - \partial_{y} f ( \mathcal{Z}_{t_{-}} )  
  \partial_{x}^{0} f ( \mathcal{Z}_{t_{-}} ) \mathd [ \mathcal{Y} ,
  \mathcal{Y} ]_{t}\\
  & = & \partial_{x}^{0} f ( \mathcal{Z}_{t_{-}} )   \partial_{y} f (
  \mathcal{Z}_{t_{-}} )   \mathd \mathcal{N}^{0}_{t} ,
\end{eqnarray*}
where we note $\mathcal{N}^{0}_{t} \assign \mathcal{N}_{t} -t$ the compensated
Poisson process, that is a martingale process.

Notice that $M_{t}^{f}$ and $N_{t}^{f}$ are not orthogonal martingales with
respect to the bracket multiplication $[ \cdot , \cdot ]$. However, recall
that the angular bracket $\langle \cdot , \cdot \rangle$, also known as the
conditional quadratic covariation (see {\tmname{Protter}} {\cite{Pro2005a}}),
is the compensator of $[ \cdot , \cdot ]$. But since $\mathcal{N}^{0}_{t}$ is
a martingale, we have for the angular bracket
\[ \mathd \langle M^{f} ,N^{f} \rangle_{t} =0, \]
that is the martingales $M_{t}^{f}$ and $N_{t}^{f}$ are orthogonal with
respect to the conditional quadratic covariation.

Similarly, the pairing of $N_{t}^{f}$ with a test martingale $M_{t}^{g}$ leads
to the quadratic covariation
\begin{eqnarray*}
  \mathd [ N^{f} ,M^{g} ]_{t} & = & \partial_{y} f ( \mathcal{Z}_{t_{-}} )
  \left[ \partial_{x}^{0} g ( \mathcal{Z}_{t_{-}} ) + ( \Delta_{t} \mathcal{X}
  )   \frac{1}{2} \partial_{x}^{2} g ( \mathcal{Z}_{t_{-}} ) \right] \mathd
  \mathcal{N}_{t}\\
  &  & \hspace{2em} - \partial^{0}_{x} f ( \mathcal{Z}_{t_{-}} )  
  \partial_{y} g ( \mathcal{Z}_{t_{-}} )   \mathd t\\
  & = & \partial_{y} f ( \mathcal{Z}_{t_{-}} ) [ ( \Delta_{t} \mathcal{X}
  )_{+}   ( \partial_{x}^{-} g ) ( \mathcal{Z}_{t_{-}} ) + ( \Delta_{t}
  \mathcal{X} )_{-}   ( \partial_{x}^{-} g ) ( \mathcal{Z}_{t_{-}} ) ] \mathd
  \mathcal{N}_{t}\\
  &  & \hspace{2em} - \partial^{0}_{x} f ( \mathcal{Z}_{t_{-}} )  
  \partial_{y} g ( \mathcal{Z}_{t_{-}} )   \mathd t,
\end{eqnarray*}
and the conditional quadratic covariation
\begin{eqnarray*}
  \mathd \langle N^{f} ,M^{g} \rangle_{t} & = & \partial_{y} f (
  \mathcal{Z}_{t_{-}} ) \left[ ( \partial_{x}^{0} g ) ( \mathcal{Z}_{t_{-}} )
  + ( \Delta_{t} \mathcal{X} )   \frac{1}{2} ( \partial_{x}^{2} g ) (
  \mathcal{Z}_{t_{-}} ) \right] \mathd \mathcal{N}_{t}\\
  &  & \hspace{2em} - \partial^{0}_{x} f ( \mathcal{Z}_{t_{-}} )  
  \partial_{y} g ( \mathcal{Z}_{t_{-}} )   \mathd t\\
  & = & \partial_{y} f ( \mathcal{Z}_{t_{-}} ) [ ( \Delta_{t} \mathcal{X}
  )_{+} ( \partial_{x}^{-} g ) ( \mathcal{Z}_{t_{-}} ) + ( \Delta_{t}
  \mathcal{X} )_{-}   ( \partial_{x}^{-} g ) ( \mathcal{Z}_{t_{-}} ) ] \mathd
  \mathcal{N}_{t}\\
  &  & \hspace{2em} - \partial^{0}_{x} f ( \mathcal{Z}_{t_{-}} )  
  \partial_{y} g ( \mathcal{Z}_{t_{-}} )   \mathd t\\
  & = & \partial_{y} f ( \mathcal{Z}_{t_{-}} ) \left[ \frac{1}{2}
  \partial_{x}^{-} g ( \mathcal{Z}_{t_{-}} ) + \frac{1}{2}   \partial_{x}^{-}
  g ( \mathcal{Z}_{t_{-}} ) \right] \mathd t\\
  &  & \hspace{2em} - \partial^{0}_{x} f ( \mathcal{Z}_{t_{-}} )  
  \partial_{y} g ( \mathcal{Z}_{t_{-}} )   \mathd t\\
  & = & \{ \partial_{y} f ( \mathcal{Z}_{t_{-}} )   \partial^{0}_{x} f (
  \mathcal{Z}_{t_{-}} ) - \partial^{0}_{x} f ( \mathcal{Z}_{t_{-}} )  
  \partial_{y} g ( \mathcal{Z}_{t_{-}} ) \}   \mathd t.
\end{eqnarray*}

{\bigskip}

\section{Proof of Theorem \ref{T: stochastic representation}}\label{S: proof
of the main result}

{\bigskip}

It remains to prove the representation formula stated in Theorem \ref{T:
stochastic representation}. Equipped with the martingale representations of
the previous sections, it suffices to follow the lines of
{\tmname{Gundy}}--{\tmname{Varopoulos}} {\cite{GunVar1979}} and
{\tmname{Arcozzi}} {\cite{Arc1995a}}. For that, let $( \mathcal{Z}_{t} )_{-
\infty <t \leqslant 0}$ be the so--called background noise. Those are
semidiscrete random walks starting at infinity in the upper--half space
$\mathbbm{Z} \times \mathbbm{R}^{+}$ and stopped at time $t=0$ when reaching
the boundary $\mathbbm{Z}$. Let $f$ defined on $\mathbbm{Z}$, $M_{t}^{f}
\assign ( P_{t} f ) ( \mathcal{Z}_{t} )$ the associated martingale, and
$N_{t}^{f}$ the corresponding martingale transform as defined previously.
Finally, introduce the projection operator
\[ \mathcal{T} f ( x ) \assign \mathbbm{E} ( N_{0}^{f} | \mathcal{Z}_{0} =x )
   . \]
It follows that for any test function $g$ defined on $\mathbbm{Z}$, and the
associated martingale $M_{t}^{g}$, we have
\begin{eqnarray*}
  ( \mathcal{T} f,g )_{\ell^{2} ( \mathbbm{Z} )} & = & \sum_{x} \mathcal{T} f ( x
  ) g ( x ) = \sum_{x} \mathbbm{E} ( N_{0}^{f} | \mathcal{Z}_{0} =x ) g ( x )
  = \sum_{x} \mathbbm{E} ( N_{0}^{f} | \mathcal{Z}_{0} =x ) M_{0}^{g}\\
  & = & \sum_{x} \mathbbm{E} ( N_{0}^{f} M_{0}^{g} | \mathcal{Z}_{0} =x ) =
  \sum_{x} \mathbbm{E} \left( \int_{- \infty}^{0} \mathd [ N^{f} ,M^{g} ]_{t} 
  \mid   \mathcal{Z}_{0} =x \right)\\
  & = & \sum_{x} \mathbbm{E} \left( \int_{- \infty}^{0} \{ \partial_{y} f (
  \mathcal{Z}_{t_{-}} )   \partial^{0}_{x} f ( \mathcal{Z}_{t_{-}} ) -
  \partial^{0}_{x} f ( \mathcal{Z}_{t_{-}} )   \partial_{y} g (
  \mathcal{Z}_{t_{-}} ) \} \mathd t  \mid   \mathcal{Z}_{0} =x \right)\\
  & = & \sum_{x} \int_{- \infty}^{0} \{ \partial_{y} f ( y,x )  
  \partial^{0}_{x} f ( y,x ) - \partial^{0}_{x} f ( y,x )   \partial_{y} g (
  y,x ) \} 2y \mathd y\\
  & = & ( \mathcal{H} f,g )
\end{eqnarray*}
where we used the fact that $M_{0}^{g}$ depends only on $\mathcal{Z}_{0}$ but
not on the trajectory, where we used the formula of the previous section for
the quadratic covariations, and finally the fact that the density of the
background noise $\mathcal{Z}_{t}$ in the upper half space is the same as in
the continuous setting, equal to $2y \mathd y$ (see {\cite{GunVar1979}}).


\begin{thebibliography}{10}

  \bibitem{Arc1995a}Nicola Arcozzi. {\newblock}\tmtextit{Riesz transforms
  on spheres and compact Lie groups}. {\newblock}ProQuest LLC, Ann Arbor, MI,
  1995. {\newblock}Thesis (Ph.D.)--Washington University in St. Louis.
  
  \bibitem{ArcDomPet2016a}
   N. Arcozzi, K. Domelevo, and S. Petermichl. {\newblock}Second Order Riesz Transforms on MultiplyÐConnected Lie Groups and Processes with Jumps. {\newblock}\tmtextit{Potential Anal.},
   45(4):777--794, 2016.  
  
  \bibitem{DomOsePet2017a}
  K. Domelevo, A. Osekowski, and S. Petermichl. {\newblock}Various sharp estimates for semi-discrete Riesz transforms of the second order.
  {\newblock}\tmtextit{ArXiv e-prints, 2017}. https://arxiv.org/abs/1701.04106.
  
  \bibitem{DomPet2014c}
  K. Domelevo and S. Petermichl. {\newblock}Sharp $L^p$ estimates for discrete second order Riesz transforms. {\newblock}\tmtextit{Adv. Math.}, 262:932--952, 2014.
  
  
  \bibitem{Duf1956a}R.~J. Duffin. {\newblock}Basic properties of discrete
  analytic functions. {\newblock}\tmtextit{Duke Math. J.}, 23:335--363, 1956.
  
  \bibitem{Ess1984}Matts Ess{\'e}n. {\newblock}A superharmonic proof of the
  M. Riesz conjugate function theorem. {\newblock}\tmtextit{Ark. Mat.},
  22(2):241--249, 1984.
  
  \bibitem{Fer1944a}Jacqueline Ferrand. {\newblock}Fonctions
  pr{\'e}harmoniques et fonctions pr{\'e}holomorphes.
  {\newblock}\tmtextit{Bull. Sci. Math. (2)}, 68:152--180, 1944.
  
  \bibitem{GohKre1970a}I.~C. Gohberg and M.~G. Kre$\breve{\text{{\i}}}$n.
  {\newblock}\tmtextit{Theory and applications of Volterra operators in
  Hilbert space}. {\newblock}Translated from the Russian by A. Feinstein.
  Translations of Mathematical Monographs, Vol. 24. American Mathematical
  Society, Providence, R.I., 1970.
  
  \bibitem{GuaMal2013a}Maru Guadie and Eugenia Malinnikova.
  {\newblock}Stability and regularization for determining sets of discrete
  Laplacian. {\newblock}\tmtextit{Inverse Problems}, 29(7):075018, 17, 2013.
  
  \bibitem{GunVar1979}Richard~F. Gundy and Nicolas~Th. Varopoulos.
  {\newblock}Les transformations de Riesz et les int{\'e}grales stochastiques.
  {\newblock}\tmtextit{C. R. Acad. Sci. Paris S{\'e}r. A-B}, 289(1):A13--A16,
  1979.
  
  \bibitem{Hei1949a}H.A. Heilbronn. {\newblock}On discrete harmonic
  functions. {\newblock}\tmtextit{Proc. Camb. Philos. Soc.}, 45:194--206,
  1949.
  
  \bibitem{Isa1941b}Rufus~Philip Isaacs. {\newblock}A finite difference
  function theory. {\newblock}\tmtextit{Rev., Ser. A, Univ. Nac. Tucuman},
  2:177--201, 1941.
  
  \bibitem{JerLevShe2014a}David Jerison, Lionel Levine, and Scott
  Sheffield. {\newblock}Internal DLA and the Gaussian free field.
  {\newblock}\tmtextit{Duke Math. J.}, 163(2):267--308, 2014.
  
  \bibitem{LipMan2015a}Gabor Lippner and Dan Mangoubi. {\newblock}Harmonic
  functions on the lattice: Absolute monotonicity and propagation of
  smallness. {\newblock}\tmtextit{Duke Math. J.}, 164(13):2577--2595, 10 2015.
  
  \bibitem{Lus2004a}Fran{\c c}oise Lust-Piquard. {\newblock}Dimension free
  estimates for discrete Riesz transforms on products of abelian groups.
  {\newblock}\tmtextit{Adv. Math.}, 185(2):289--327, 2004.
  
  \bibitem{MatSod2000a}V.~Matsaev and M.~Sodin. {\newblock}Distribution of
  Hilbert transforms of measures. {\newblock}\tmtextit{Geom. Funct. Anal.},
  10(1):160--184, 2000.
  
  \bibitem{Mat1961a}Vladimir Matsaev. {\newblock}Volterra operators
  obtained from self-adjoint operators by perturbation.
  {\newblock}\tmtextit{Dokl. Akad. Nauk SSSR}, 139:810--813, 1961.
  
  \bibitem{Pic1972}S.~K. Pichorides. {\newblock}On the best values of the
  constants in the theorems of M. Riesz, Zygmund and Kolmogorov.
  {\newblock}\tmtextit{Studia Math.}, 44:165--179. (errata insert), 1972.
  {\newblock}Collection of articles honoring the completion by Antoni Zygmund
  of 50 years of scientific activity, II.
  
  \bibitem{Pri2009a}Nicolas Privault. {\newblock}\tmtextit{Stochastic
  analysis in discrete and continuous settings with normal martingales},
  volume 1982 of \tmtextit{Lecture Notes in Mathematics}.
  {\newblock}Springer-Verlag, Berlin, 2009.
  
  \bibitem{Pri2014a}Nicolas Privault. {\newblock}\tmtextit{Stochastic
  finance}. {\newblock}Chapman \& Hall/CRC Financial Mathematics Series. CRC
  Press, Boca Raton, FL, 2014. {\newblock}An introduction with market
  examples.
  
  \bibitem{Pro2005a}Philip~E. Protter. {\newblock}\tmtextit{Stochastic
  integration and differential equations}, volume~21 of \tmtextit{Stochastic
  Modelling and Applied Probability}. {\newblock}Springer-Verlag, Berlin,
  2005. {\newblock}Second edition. Version 2.1, Corrected third printing.
  
  \bibitem{Ver1980a}Igor~E. Verbitsky. {\newblock}Estimate of the norm of
  a function in a Hardy space in terms of the norms of its real and imaginary
  parts. {\newblock}\tmtextit{Mat. Issled.}, (54):16--20, 164--165, 1980.
  {\newblock}English transl.: Amer. Math. Soc. Transl. (2) 124 (1984), 11-15.
\end{thebibliography}
\end{document}